\documentclass[12pt,reqno]{article}

\usepackage{amsfonts}
\usepackage{mathrsfs}

\usepackage{amsmath,amsthm,amsfonts,amssymb,latexsym,mathrsfs,color}

\usepackage[colorlinks=true,
linkcolor=webblue, filecolor=webbrown, citecolor=webred ]{hyperref}
\definecolor{webblue}{rgb}{0,.5,0}
\definecolor{webred}{rgb}{0,.5,0}
\definecolor{webbrown}{rgb}{.6,0,0}

\newtheorem{thm}{Theorem}[section]
\newtheorem{lem}[thm]{Lemma}

\newtheorem{prop}[thm]{Proposition}
\newtheorem{conj}[thm]{Conjecture}

\newtheorem{ex}{Example}[section]

\newtheorem{cl}[thm]{Claim}
\newtheorem{rem}[thm]{Remark}
\newtheorem{prob}[thm]{Problem}

\numberwithin{equation}{section}

\newcommand{\D}{\displaystyle}
\newcommand{\DF}[2]{\frac{\D#1}{\D#2}}

\title{Positivity of iterated sequences of polynomials
\thanks{Supported partially by the National Natural Science Foundation of China (No. 11571150).
\newline\hspace*{5mm}
   {\it Email address:} bxzhu@jsnu.edu.cn (B.-X. Zhu)}}
\author{Bao-Xuan Zhu}
\date{\footnotesize School of Mathematics and Statistics,
         Jiangsu Normal University,
         Xuzhou 221116, PR China}

\begin{document}

\maketitle

\begin{abstract}
In this paper, we present some criteria for the
$2$-$q$-log-convexity and $3$-$q$-log-convexity of combinatorial
sequences, which can be regarded as the first column of certain
infinite triangular array $[A_{n,k}(q)]_{n,k\geq0}$ of polynomials
in $q$ with nonnegative coefficients satisfying the recurrence
relation
$$A_{n,k}(q)=A_{n-1,k-1}(q)+g_k(q)A_{n-1,k}(q)+h_{k+1}(q)A_{n-1,k+1}(q).$$
Those criterions can also be presented by continued fractions and
generating functions.  These allow a unified treatment of the
$2$-$q$-log-convexity of alternating Eulerian polynomials,
$2$-log-convexity of Euler numbers, and $3$-$q$-log-convexity of
many classical polynomials, including the Bell polynomials, the
Eulerian polynomials of Types $A$ and $B$, the $q$-Schr\"{o}der
numbers, $q$-central Delannoy numbers, the Narayana polynomials of
Types $A$ and $B$, the generating functions of rows in the Catalan
triangles of Aigner and Shapiro,  the generating functions of rows
in the large Schr\"oder triangle, and so on, which extend many known
results for $q$-log-convexity.
\end{abstract}

\section{Introduction}
\subsection{Notation}
Let $\{a_n\}_{n\geq0}$ be a sequence of nonnegative real numbers.
The sequence is called {\it log-concave} (resp. {\it log-convex}) if
for all $n\ge 1$, $a_{n-1}a_{n+1}\le a_n^2$ (resp.
$a_{n-1}a_{n+1}\ge a_n^2$). The log-convex and log-concave sequences
arise often in combinatorics, algebra, geometry, analysis,
probability and statistics and have been extensively investigated,
see Stanley~\cite{Sta89} and Brenti~\cite{Bre94} for log-concavity,
Liu and Wang~\cite{LW07} for log-convexity.

In recent years there has been a growing interest in the
$q$-log-concavity and q-log-convexity of $q$-analogs of
combinatorial sequences. Recall their definitions. For two
polynomials with real coefficients $f(q)$ and $g(q)$, denote
$f(q)\geq_q g(q)$ if the difference
$f\left(q\right)-g\left(q\right)$ has only nonnegative coefficients.
For a polynomial sequence $\{f_n(q)\}_{n\geq 0}$, it is called {\it
$q$-log-concave} suggested by Stanley if
$$f_n(q)^2\geq_q f_{n+1}(q)f_{n-1}(q)$$ for $n\geq 1$.
It is called {\it $q$-log-convex} defined by Liu and Wang if
$$ f_{n+1}(q)f_{n-1}(q)\geq_q f_n(q)^2$$ for $n\geq 1$. It was proved that many famous polynomials have the $q$-log-concavity or the $q$-log-convexity,
e.g., $q$-binomial coefficients and $q$-Stirling numbers of two
kinds~\cite{But90,Le90,Sag921}, the Bell polynomials~\cite{LW07},
the classical Eulerian polynomials~\cite{LW07,Zhu13}, the Narayana
polynomials of type $A$~\cite{CTWY10}, the Narayana polynomials of
type $B$~\cite{CWY10,Zhu13}, the generating functions of
Jacobi-Stirling numbers~\cite{LZ,Zhu14}, the Bessel polynomials
\cite{CWY11}, the Ramanujan polynomials \cite{CWY11}, and so on.

Motivated by the notion of infinite log-concavity \cite{Moll02} and
infinite $q$-log-concavity \cite{MS08}, as an extension of
$q$-log-convexity, Chen \cite{Chen08} defined the infinite
$q$-log-convexity as follows. Define the operator $\mathcal {L}$
which maps a polynomial sequence $\{f_i(q)\}_{i\geq 0}$ to a
polynomial sequence $\{g_i(q)\}_{i\geq 1}$ given by
$$g_i(q):=f_{i-1}(q)f_{i+1}(q)-f_i(q)^2.$$
Then the $q$-log-convexity of $\{f_i(q)\}_{i\geq 0}$ is equivalent
to the $q$-positivity of $\mathcal {L}\{f_i(q)\}$, i.e., the
coefficients of $g_i(q)$ are nonnegative for all $i\geq1$. In
general, we say that $\{f_i(q)\}_{i\geq 0}$ is {\it
$k$-$q$-log-convex} if the coefficients of $\mathcal
{L}^m\{f_i(q)\}$ are nonnegative for all $m\leq k$, where $\mathcal
{L}^m=\mathcal {L}(\mathcal {L}^{m-1})$. It is called {\it
infinitely $q$-log-convex}  if $\{f_i(q)\}_{i\geq 0}$ is
$k$-$q$-log-convex for every $k \geq 0$. If we take $q=0$, then the
polynomial sequence reduces to a sequence of real numbers. So the
$k$-$q$-log-convexity turns out to be $k$-log-convexity, see Chen
and Xia \cite{CX11}.

\subsection{Motivations}

In \cite{BM04}, Boros and Moll studied the quartic integral
$$\int_0^{\infty}\frac{1}{(t^4+2xt^2+1)^{n+1}}dt$$
and got the following formula for any $x>-1$ and any nonnegative
integer $n$ that
$$\int_0^{\infty}\frac{1}{(t^4+2xt^2+1)^{n+1}}dt=\frac{\pi}{2^{n+3/2}(x+1)^{n+1/2}}P_n(x),$$
where
$$P_n(x)=\sum_{j,k}\binom{2n+1}{2j}\binom{n-j}{k}\binom{2k+2j}{k+j}\frac{(x+1)^j(x-1)^k}{2^{3(k+j)}}$$ are called the Boros-Moll
polynomials. Moll \cite{Moll02} conjectured that the coefficients of
$P_n(x)$ form a log-concave sequence, which was proved by Kauers and
Paule \cite{KP07}. Moreover, as an extension, Chen {\it et.al}
\cite{CX13,CDY13} proved that $2$-log-concavity and
$3$-log-concavity of the coefficients of $P_n(x)$ and confirmed two
Br\"{a}nd\'{e}n's conjectures on real roots of polynomials related
to $P_n(x)$ \cite{Bra11}. However, the stronger conjecture on the
infinitely log-concavity of the coefficients of $P_n(x)$ proposed by
Boros and Moll \cite{BM04} is still open. On the other hand, Chen
\cite{Chen08} also proposed the next open conjecture on the infinite
$q$-log-convexity.
\begin{conj} Boros-Moll polynomials are infinitely $q$-log-convex. \end{conj}
 In
general, it is much more difficult to show the $k$-$q$-log-convexity
for $k\geq2$. So far, there has been no nontrivial example. Thus,
motivated by these, we will study the higher order $q$-log-convexity
in this paper.

Let $\sigma=(g_k(q))_{k\ge 0}$ and $\tau=(h_{k+1}(q))_{k\ge 0}$ be
two sequences of polynomials. Define an infinite lower triangular
matrix $[A_{n,k}(q)]_{n,k\geq0}$ satisfies the recurrence
\begin{eqnarray}\label{re}
A_{0,0}(q)=1,\quad
A_{n,k}(q)&=&A_{n-1,k-1}(q)+g_k(q)A_{n-1,k}(q)+h_{k+1}(q)A_{n-1,k+1}(q),
\end{eqnarray}
where $A_{n,k}(q)=0$ unless $n\ge k\ge 0$. Many well-known
polynomials can be viewed as the first column $A_{n,0}(q)$ of
$[A_{n,k}(q)]_{n,k\geq0}$. The following are some classical
examples, e.g., see \cite{Aig01,WZ16}.

\begin{ex}\label{basic-qSM}
\end{ex}
\begin{itemize}
\item [\rm (i)]
If $g_k=k+q$ and $h_{k}=kq$, then the first column $A_{n,0}(q)$ for
$n\geq0$ equal to the Bell polynomials, which are the generating
functions $B_n(q)=\sum_{k=0}^{n}S(n,k)q^k$ of the Stirling numbers
of the second kind. They can be looked as a $q$-analog of the Bell
numbers and have many fascinating properties (see \cite[\S
4.1.3]{Rom84} for instance). In addition, it is well known that the
Bell polynomials $B_n(q)$ have only real zeros and $S(n,k)$ is
therefore log-concave in $k$ for each fixed $n$ (see \cite{WYjcta05}
for instance).
\end{itemize}

\begin{itemize}
\item [\rm (ii)]
If $g_k=(k+1)q+k$ and $h_{k}=k^2q$, then the first column
$A_{n,0}(q)$ for $n\geq0$ equal to the Eulerian polynomials defined
by the descent statistics. Let $\pi=a_1a_2\cdots a_n$ be a
permutation of $[n]$. An element $i\in [n-1]$ is called a descent of
$\pi$ if $a_i>a_{i+1}$. Denote by $A(n,k)$ the number of
permutations of $[n]$ having $k-1$ descents, which is called the
Eulerian number. Then the Eulerian polynomials
$E_n(q)=\sum_{k=0}^nE(n,k)q^k$ for $n\geq0$. It is well known that
$A_n(q)$ has only real zeros and $A(n,k)$ is therefore log-concave
in $k$ for each fixed $n$ (see \cite{WYjcta05} for instance).
\end{itemize}

\begin{itemize}
\item [\rm (iii)]
If $g_0=q+1,g_k=2q+1$ and  $h_k=q(q+1)$, then the first column
$A_{n,0}(q)$ for $n\geq0$ equal to the $q$-Schr\"oder numbers
$r_n(q)=\sum_{k=0}^n\binom{n+k}{n-k}\frac{1}{k+1}\binom{2k}{k}q^k$~\cite{BSS93}.
They are defined as the $q$-analog of the large Schr\"oder numbers
$r_n$:
$$r_n(q)=\sum_{P}q^{\mathrm{diag}(P)},$$
where $P$ takes over all Schr\"oder paths from $(0,0)$ to $(n,n)$
and $\mathrm{diag} (P)$ denotes the number of diagonal steps in the
path $P$.\end{itemize}

\begin{itemize}
\item [\rm (iv)]

If $g_k=1+2q,h_1=2q(1+q)$ and $h_k=q(1+q)$, then the first column
$A_{n,0}(q)$ for $n\geq0$ equal to the $q$-central Delannoy numbers
$D_n(q)=\sum_{k=0}^n\binom{n+k}{n-k}\binom{2k}{k}q^k$ \cite{Sag98},
which reduce to famous central Delannoy numbers for $q=1$.
\end{itemize}

\begin{itemize}
\item [\rm (v)]
If $g_0=q, g_k=1+q$ and
  $h_k=q$, then the first column
$A_{n,0}(q)$ for $n\geq0$ equal to the Narayana polynomials
$N_n(q)=\sum_{k=1}^n\frac{1}{n}\binom{n}{k}\binom{n}{k-1}q^k$, where
$\frac{1}{n}\binom{n}{k}\binom{n}{k-1}$ is called the Narayana
number, which is defined as the number of Dyck paths of length $2n$
with exactly $k$ peaks. In addition, the Narayana polynomials
$N_n(q)$ are closely related to the $q$-Schr\"oder numbers $r_n(q)$
by $r_n(q)=N_n(1+q)$ \cite{Sul02}. It was  also proved in
\cite{ZS15} that the Narayana transformation preserves strong
$q$-log-convexity of polynomials.
\end{itemize}

\begin{itemize}
\item [\rm (vi)]
If $g_k=1+q, h_1=2q$ and $h_k=q$ for
  $k>1$, then the first column
$A_{n,0}(q)$ for $n\geq0$ equal to the Narayana polynomials
$W_n(q)=\sum_{k=0}^n\binom{n}{k}^2q^k$ of type $B$. They appeared as
the rank generating function of the lattice of noncrossing
partitions of type $B$ on $[n]=\{1, 2,...,n\}$ \cite{Re92} and are
also the coordinator polynomial of the growth series for the
classical root lattice $A_n$ \cite{ABHPS11,Be83}. Moreover, the
Narayana transformation of Type $B$ preserves strong
$q$-log-convexity of polynomials \cite{ZS15}.
\end{itemize}

The organization of this paper is as follows. In Section 2, we state
our main results and apply total positivity of matrices to give the
proof of Theorem \ref{thm+2+q+log-convex}. In Section 3, we apply
Theorem \ref{thm+2+q+log-convex} to sequences of real numbers and
Riordan array. In particular, we get $3$-$q$-log-convexity of the
generating functions of rows in the Catalan triangles of Aigner and
Shapiro, and the large Schr\"oder triangle. In Section 4, we apply
Theorem \ref{thm+q+continued+q} to some classical polynomials in a
unified manner, and get the $2$-$q$-log-convexity of alternating
Eulerian polynomials, $2$-log-convexity of Euler numbers, and
$3$-$q$-log-convexity of many classical polynomials, including the
Bell polynomials, the Eulerian polynomials of Types $A$ and $B$, the
$q$-Schr\"{o}der numbers, $q$-central Delannoy numbers, the Narayana
polynomials of Types $A$ and $B$,  and so on. These applications
extend and strengthen many known results for $q$-log-convexity or
log-convexity. Finally, in Section 5, as the concluding remarks, we
will point out some further research problems and conjectures.

\section{Main results}
 In this section, we prove the next criterions for
$2$-$q$-log-convexity and $3$-$q$-log-convexity, respectively.

\begin{thm}\label{thm+2+q+log-convex}
Let $\{g_n(q)\}_{n\geq0}$ and $\{h_n(q)\}_{n\geq1}$ be two sequences
of polynomials with nonnegative coefficients. Assume that the
triangular array $[A_{n,k}(q)]_{n,k\geq0}$ is defined in (\ref{re}).
Then we have the following:
\begin{itemize}
\item [\rm (i)]
If $g_k(q)g_{k+1}(q)g_{k+2}(q)-
h_{k+1}(q)g_{k+2}(q)-g_{k}(q)h_{k+2}(q)\geq_q 0$ for $k\geq0$, then
the first column sequence $\{A_{n,0}(q)\}_{n\geq 0}$ is
$2$-$q$-log-convex.
\item [\rm (ii)]
If $g_k(q)g_{k+1}(q)\geq_q h_{k+1}(q)$ and
\begin{eqnarray*}
&&g_k(q)g_{k+1}(q)g_{k+2}(q)g_{k+3}(q)
-g_{k+2}(q)g_{k+3}(q)h_{k+1}(q)-g_k(q)g_{k+3}(q)h_{k+2}(q)-\nonumber\\
&& g_k(q)g_{k+1}(q)h_{k+3}(q) + h_{k+1}(q)h_{k+3}(q)\geq_q 0
\end{eqnarray*} for $k\geq0$, then the first column sequence
$\{A_{n,0}(q)\}_{n\geq 0}$ is $3$-$q$-log-convex.
\end{itemize}
\end{thm}

In combinatorics, one often meets the continued fractions. In fact,
the sequence $\{A_{n,0}(q)\}_{n\geq0}$ in (\ref{re}) is also closely
related to the famous Jacobi continued fractions. If let its
generating function
$$G(x)=\sum_{n\geq 0}A_{n,0}(q)x^n,$$
then it can be expressed by the Jacobi continued fraction as follows
\begin{eqnarray*} G(x)=\DF{1}{1- s_0x-\DF{t_1x^2}{1- s_1x-\DF{t_2x^2}{1-
s_2x-\ldots}}},
\end{eqnarray*}
see \cite{W48} or \cite{Fla80} for instance. So, Theorem
\ref{thm+2+q+log-convex} can also be restated by the continued
fraction as follows.

\begin{thm}\label{thm+q+continued+q}
Given two sequences $\{g_k(q)\}_{k\geq 0}$ and $\{h_k(q)\}_{k\geq
1}$ of polynomials with nonnegative coefficients, let
\begin{eqnarray*}
\sum\limits_{n=0}^{\infty}T_n(q) x^n=\DF{1}{1-
g_0(q)x-\DF{h_1(q)x^2}{1- g_1(q)x-\DF{h_2(q)x^2}{1-
g_2(q)x-\ldots}}}.
\end{eqnarray*}
Then we have the following:
\begin{itemize}
\item [\rm (i)]
If $g_k(q)g_{k+1}(q)g_{k+2}(q)-
h_{k+1}(q)g_{k+2}(q)-g_{k}(q)h_{k+2}(q)\geq_q 0$ for $k\geq0$, then
the sequence $\{T_n(q)\}_{n\geq 0}$ is $2$-$q$-log-convex.
\item [\rm (ii)]
If $g_k(q)g_{k+1}(q)\geq_q h_{k+1}(q)$ and
\begin{eqnarray*}
&&g_k(q)g_{k+1}(q)g_{k+2}(q)g_{k+3}(q)
-g_{k+2}(q)g_{k+3}(q)h_{k+1}(q)-g_k(q)g_{k+3}(q)h_{k+2}(q)-\nonumber\\
&& g_k(q)g_{k+1}(q)h_{k+3}(q) + h_{k+1}(q)h_{k+3}(q)\geq_q 0
\end{eqnarray*} for $k\geq0$, then the sequence
$\{T_n(q)\}_{n\geq 0}$ is $3$-$q$-log-convex.
\end{itemize}
\end{thm}
\subsection{Proof of Theorem \ref{thm+2+q+log-convex}}

Total positivity of matrices plays an important role in our proof.
Therefore let us recall the definition. Let $M=[m_{n,k}]_{n,k\ge 0}$
be a matrix of real numbers. It is called {\it totally positive}
({\it TP} for short) if all its minors are nonnegative and is called
{\it TP$_r$} if all minors of order $\le r$ are nonnegative. When
each entry of $M$ is a polynomial in $q$ with nonnegative
coefficients, then we have the similar concepts for $q$-TP (resp.
$q$-TP$_r$) if all its minors (resp. if all minors of order $\le r$)
are polynomials with nonnegative coefficients. Total positivity of
matrices plays an important role in various branches of mathematics,
statistics, probability, mechanics, economics, and computer science,
see Karlin \cite{Kar68} and Pinkus \cite{Pin10} for instance. Theory
of total positivity has successfully been applied to log-concavity
problems in combinatorics, see Brenti \cite{Bre95,Bre96}.

In order to present our proof, we need the following some basic
results from total positivity of matrices. The first of the
following two lemmas is direct by the definition and the second
follows from the classic Cauchy-Binet formula.

\begin{lem}\label{lps-lem}
A matrix is $q$-TP$_r$ if and only if its leading principal
submatrices are all $q$-TP$_r$.
\end{lem}
\begin{lem}\label{prod-lem}
If two matrices are $q$-TP$_r$, then so is their product.
\end{lem}
Using the total positivity of matrices, we have the following
criterion for higher order $q$-log-convexity.
\begin{prop}\label{prop+higher+q-log-concave}
Let $\{a_n(q)\}_{n\geq0}$ be a sequence of  polynomials with
nonnegative coefficients. If the Hankel matrix
$[a_{i+j}(q)]_{i,j\geq0}$ is q-TP$_{r+1}$, then
$\{a_n(q)\}_{n\geq0}$ is $r$-q-log-convex for $1\leq r\leq 3$.
\end{prop}
\begin{proof}
For brevity, we write $a_k$ for $a_k(q)$. Note for $r=1$ that
\begin{eqnarray*}\mathcal {L}(a_k)=a_{k+1}a_{k-1}-a_k^2=\left|\begin{array}{cc}
a_{k-1}&a_{k}\\
a_{k}&a_{k+1}
\end{array}\right|.
\end{eqnarray*} Thus it is obvious that $\{a_n(q)\}_{n\geq0}$ is q-log-convex if the Hankel matrix $[a_{i+j}]_{i,j\geq0}$
is q-TP$_{2}$. Furthermore, for $r=2$, we get that
\begin{eqnarray}\label{2-log}
\mathcal {L}^2(a_k)&=&\mathcal {L}(a_{k-1})\mathcal
{L}(a_{k+1})-\left[\mathcal {L}(a_k)\right]^2\nonumber\\
&=&\left(a_{k+2}a_{k}-a_{k+1}^2\right)\left(a_{k}a_{k-2}-a_{k-1}^2\right)-\left(a_{k+1}a_{k-1}-a_k^2\right)^2\nonumber\\
&=&a_k\left(2a_{k-1}a_{k}a_{k+1}+a_{k}a_{k+2}a_{k-2}-a_k^3-a_{k+1}^2a_{k-2}-a_{k-1}^2a_{k+2}\right)\nonumber\\
&=&a_k\left|\begin{array}{ccc}
a_{k-2}&a_{k-1}&a_{k}\\
a_{k-1}&a_k&a_{k+1}\\
a_k&a_{k+1}&a_{k+2}
\end{array}\right|,
\end{eqnarray}
which implies that if the Hankel matrix $[a_{i+j}(q)]_{i,j\geq0}$ is
q-TP$_{3}$ then $\{a_n(q)\}_{n\geq0}$ is $2$-q-log-convex.

In the following, we proceed to consider the case for $r=3$. By
(\ref{2-log}), we have
\begin{eqnarray*}
\mathcal {L}^3(a_k)&=&\mathcal {L}^2(a_{k-1})\mathcal
{L}^2(a_{k+1})-\left[\mathcal {L}^2(a_k)\right]^2\\
&=&a_{k+1}a_{k-1}\left|\begin{array}{ccc}
a_{k-1}&a_{k}&a_{k+1}\\
a_{k}&a_{k+1}&a_{k+2}\\
a_{k+1}&a_{k+2}&a_{k+3}\end{array}\right|
 \left|\begin{array}{ccc}
a_{k-3}&a_{k-2}&a_{k-1}\\
a_{k-2}&a_{k-1}&a_{k}\\
a_{k-1}&a_{k}&a_{k+1}
\end{array}\right|-a^2_k\left|\begin{array}{ccc}
a_{k-2}&a_{k-1}&a_{k}\\
a_{k-1}&a_k&a_{k+1}\\
a_k&a_{k+1}&a_{k+2}
\end{array}\right|^2\\
&=&(a_{k+1}a_{k-1}-a^2_k)\times\\
&&\left(a^2_k\left|\begin{array}{cccc}
a_{k-3}&a_{k-2}&a_{k-1}&a_{k}\\
a_{k-2}&a_{k-1}&a_{k}&a_{k+1}\\
a_{k-1}&a_{k}&a_{k+1}&a_{k+2}\\
a_k&a_{k+1}&a_{k+2}&a_{k+3}
\end{array}\right|+\left|\begin{array}{ccc}
a_{k-3}&a_{k-2}&a_{k-1}\\
a_{k-2}&a_{k-1}&a_{k}\\
a_{k-1}&a_{k}&a_{k+1}
\end{array}\right|\left|\begin{array}{ccc}
a_{k-1}&a_{k}&a_{k+1}\\
a_{k}&a_{k+1}&a_{k+2}\\
a_{k+1}&a_{k+2}&a_{k+3}
\end{array}\right|\right).
\end{eqnarray*}
So if the Hankel matrix $[a_{i+j}(q)]_{i,j\geq0}$ is q-TP$_{4}$ then
$\{a_n(q)\}_{n\geq0}$ is $3$-q-log-convex. This completes the proof.
\end{proof}
For the sequence $\{a_n(q)\}_{n\geq0}$ in
Proposition~\ref{prop+higher+q-log-concave}, if let $q=0$, then
Proposition~\ref{prop+higher+q-log-concave} can be reduced to the
following result for sequences of real numbers.
\begin{prop}
Let $\{a_n\}_{n\geq0}$ be a sequence of nonnegative real numbers. If
the Hankel matrix $[a_{i+j}]_{i,j\geq0}$ is TP$_{r+1}$, then
 $\{a_n\}_{n\geq0}$ is $r$-log--convex for $1\leq r\leq 3$.
\end{prop}

\begin{lem}\label{lem+diag+positivity}Given two sequences $\{g_k(q)\}_{k\geq 0}$ and $\{h_k(q)\}_{k\geq
1}$ of polynomials with nonnegative coefficients, assume that the
matrix
\[\mathcal {J}_n(q)=(J_{i,j})_{0\leq i,j\leq n-1}=\left(\begin{array}{ccccc}
g_0(q)&1&0&\cdots&0\\
h_1(q)&g_1(q)&1&\cdots&0\\
0&h_2(q)&g_2(q)&\cdots&0\\
\vdots&\vdots&\vdots&\cdots&\vdots\\
0&0&0&\cdots&g_{n-1}(q)
\end{array}\right).
\]
Then we have the following.
\begin{itemize}
\item [\rm (i)]
$\mathcal {J}_n(q)$ is $q$-TP$_3$ if and only if \begin{eqnarray*}
&&g_k(q)g_{k+1}(q)g_{k+2}(q)-
h_{k+1}(q)g_{k+2}(q)-g_{k}(q)h_{k+2}(q)\geq_q 0\label{TP3}
\end{eqnarray*}
for $k\geq0$.
\item [\rm (ii)]
$\mathcal {J}_n(q)$ is $q$-TP$_4$ if and only if
$g_k(q)g_{k+1}(q)\geq_q h_{k+1}(q)$ and
\begin{eqnarray*} &&g_k(q)g_{k+1}(q)g_{k+2}(q)g_{k+3}(q)
-g_{k+2}(q)g_{k+3}(q)h_{k+1}(q)-g_k(q)g_{k+3}(q)h_{k+2}(q)\nonumber\\
&&- g_k(q)g_{k+1}(q)h_{k+3}(q) + h_{k+1}(q)h_{k+3}(q)\geq_q0.
\end{eqnarray*}
\end{itemize}
\end{lem}
\begin{proof} (i) Note that $\mathcal {J}_n(q)$ is $q$-TP$_2$ if and only if
$$g_k(q)g_{k+1}(q)\geq_q h_{k+1}(q)$$ for $k\geq0$. In addition, it is not hard to find
that $\mathcal {J}_n(q)$ is $q$-TP$_3$ if and only if
\begin{eqnarray*}
&&g_k(q)g_{k+1}(q)\geq_q h_{k+1}(q),\label{TP2}\\
&&g_k(q)g_{k+1}(q)g_{k+2}(q)-
h_{k+1}(q)g_{k+2}(q)-g_{k}(q)h_{k+2}(q)\geq_q 0\label{TP3}
\end{eqnarray*}
for $k\geq0$. Note that
\begin{eqnarray*}
&&g_k(q)g_{k+1}(q)g_{k+2}(q)-
h_{k+1}(q)g_{k+2}(q)-g_{k}(q)h_{k+2}(q)\\
&=&g_{k+2}(q)[g_k(q)g_{k+1}(q)- h_{k+1}(q)]-g_{k}(q)h_{k+2}(q).
\end{eqnarray*}
Thus if $$ g_k(q)g_{k+1}(q)g_{k+2}(q)-
h_{k+1}(q)g_{k+2}(q)-g_{k}(q)h_{k+2}(q)\geq_q 0 ,$$ then we must
have
$$g_k(q)g_{k+1}(q)\geq_q h_{k+1}(q).$$
So it follows from
\begin{eqnarray*}
&&g_k(q)g_{k+1}(q)g_{k+2}(q)-
h_{k+1}(q)g_{k+2}(q)-g_{k}(q)h_{k+2}(q)\geq_q 0\label{TP3}
\end{eqnarray*}
for $k\geq0$ that $\mathcal {J}_n(q)$ is $q$-TP$_3$. This completes
the proof of (i).

(ii)
 Note
that $\mathcal {J}_n(q)$ is $q$-TP$_4$ if and only if $\mathcal
{J}_n(q)$ is $q$-TP$_3$ and for all $k \geq 0$ that
\begin{eqnarray*}
&&g_k(q)g_{k+1}(q)g_{k+2}(q)g_{k+3}(q)
-g_{k+2}(q)g_{k+3}(q)h_{k+1}(q)-g_k(q)g_{k+3}(q)h_{k+2}(q)\nonumber\\
&&- g_k(q)g_{k+1}(q)h_{k+3}(q) + h_{k+1}(q)h_{k+3}(q)\geq_q
0.\label{TP4}
\end{eqnarray*}

Thus, if
$$g_k(q)g_{k+1}(q)h_{k+3}(q)-h_{k+1}(q)h_{k+3}(q)=[g_k(q)g_{k+1}(q)-h_{k+1}(q)]h_{k+3}(q)\geq_q0$$
and
\begin{eqnarray*}
&&g_k(q)g_{k+1}(q)g_{k+2}(q)g_{k+3}(q)
-g_{k+2}(q)g_{k+3}(q)h_{k+1}(q)-g_k(q)g_{k+3}(q)h_{k+2}(q)\nonumber\\
&&- g_k(q)g_{k+1}(q)h_{k+3}(q) + h_{k}(q)h_{k+3}(q)\geq_q0,
\end{eqnarray*} then we have $$g_k(q)g_{k+1}(q)g_{k+2}(q)
-g_{k+2}(q)h_{k+1}(q)-g_k(q)h_{k+2}(q)\geq_q0$$ since
\begin{eqnarray*}
&&g_k(q)g_{k+1}(q)g_{k+2}(q)g_{k+3}(q)
-g_{k+2}(q)g_{k+3}(q)h_{k+1}(q)-g_k(q)g_{k+3}(q)h_{k+2}(q)\nonumber\\
&&- g_k(q)g_{k+1}(q)h_{k+3}(q) + h_{k+1}(q)h_{k+3}(q)\\
&=&g_{k+3}(q)[g_k(q)g_{k+1}(q)g_{k+2}(q)
-g_{k+2}(q)h_{k+1}(q)-g_k(q)(q)h_{k+2}(q)]- g_k(q)g_{k+1}(q)h_{k+3}(q)\nonumber\\
&& + h_{k+1}(q)h_{k+3}(q).
\end{eqnarray*}

Thus, if $g_k(q)g_{k+1}(q)\geq_q h_{k+1}(q)$ and
\begin{eqnarray*} &&g_k(q)g_{k+1}(q)g_{k+2}(q)g_{k+3}(q)
-g_{k+2}(q)g_{k+3}(q)h_{k+1}(q)-g_k(q)(q)g_{k+3}(q)h_{k+2}(q)\nonumber\\
&&- g_k(q)g_{k+1}(q)h_{k+3}(q) + h_{k+1}(q)h_{k+3}(q)\geq_q0,
\end{eqnarray*} then $\mathcal {J}_n(q)$ is $q$-TP$_4$.
The proof of (ii) is complete.
\end{proof}
\textbf{Proof of Theorem \ref{thm+2+q+log-convex}:} Assume that the
matrices
\[\mathcal {Q}_n(q)=\left(\begin{array}{ccccc}
A_{1,0}(q)&A_{1,1}(q)&0&\cdots&0\\
A_{2,0}(q)&A_{2,1}(q)&A_{2,2}(q)&\cdots&0\\
A_{3,0}(q)&A_{3,1}(q)&A_{3,2}(q)&\cdots&0\\
\vdots&\vdots&\vdots&\cdots&\vdots\\
A_{n,0}(q)&A_{n,1}(q)&A_{n,2}(q)&\cdots&A_{n,n-1}(q)\\
\end{array}\right)
\]
and
\[\mathcal {A}_n(q)=(A_{i,j}(q))_{0\leq i,j\leq n-1}=\left(\begin{array}{ccccc}
A_{0,0}(q)&0&0&\cdots&0\\
A_{1,0}(q)&A_{1,1}(q)&0&\cdots&0\\
A_{2,0}(q)&A_{2,1}(q)&A_{2,2}(q)&\cdots&0\\
\vdots&\vdots&\vdots&\cdots&\vdots\\
A_{n-1,0}(q)&A_{n-1,1}(q)&A_{n-1,2}(q)&\cdots&A_{n-1,n-1}(q)\\
\end{array}\right).
\]
So by (\ref{re}) it is easy to deduce that
\begin{eqnarray}\label{eq}
\mathcal {Q}_n(q)=\mathcal {A}_n(q) \mathcal {J}_n(q).
\end{eqnarray}

Let
$$\mathcal {H}=\left[A_{n+m,0}(q)\right]_{n,m\ge 0}=
\left[
  \begin{array}{cccc}
    A_{0,0}(q) & A_{1,0}(q) & A_{2,0}(q) & \cdots \\
    A_{1,0}(q) & A_{2,0}(q) & A_{3,0}(q) & \cdots \\
    A_{2,0}(q) & A_{3,0}(q) & A_{4,0}(q) & \cdots \\
    \vdots & \vdots & \vdots & \ddots \\
  \end{array}
\right]$$ be the Hankel matrix of the sequence
$\{A_{n,0}(q)\}_{n\geq0}$ and the infinite matrix \[\mathcal
{A}(q)=(A_{i,j}(q))=\left(\begin{array}{ccccc}
A_{0,0}(q)&0&0&\cdots\\
A_{1,0}(q)&A_{1,1}(q)&0&\cdots\\
A_{2,0}(q)&A_{2,1}(q)&A_{2,2}(q)&\cdots\\
\vdots&\vdots&\vdots&\ddots\\
\end{array}\right).
\]

\begin{cl}\label{cl1}
If $\mathcal {J}_n(q)$ is $q$-TP$_r$, then the Hankel matrix
$\mathcal {H}$ is $q$-TP$_r$.
\end{cl}
\begin{proof}
If $\mathcal {J}_n(q)$ is $q$-TP$_r$, then, by induction on $n$, we
conclude that $\mathcal {A}_n(q)$ is $q$-TP$_r$ by (\ref{eq}) and
Lemma \ref{prod-lem}. So is $\mathcal {A}(q)$ by Lemma
\ref{lps-lem}. From Aigner's Fundamental Theorem in \cite{Aig01}, we
have $\mathcal {H}=\mathcal {A}(q)\mathcal {T}\mathcal {A}(q)'$,
where $\mathcal {T}={\rm diag}(T_0,T_1,T_2,T_3,\ldots)$ and $T_0=1,
T_n=T_{n-1}h_n(q)$ for $n\geq1$. Thus this again implies that the
Hankel matrix $\mathcal {H}$ is $q$-TP$_r$ by Lemma \ref{prod-lem}.
This proves the claim.
\end{proof}

We start to prove that (i) holds. Since
$$g_k(q)g_{k+1}(q)g_{k+2}(q)-
h_{k+1}(q)g_{k+2}(q)-g_{k}(q)h_{k+2}(q)\geq_q 0$$ for $k\geq0$, it
follows from Lemma \ref{lem+diag+positivity} (i) and Claim \ref{cl1}
that the Hankel matrix $\mathcal {H}$ is $q$-TP$_3$. So it follows
from Proposition \ref{prop+higher+q-log-concave} that
$\{A_{n,0}(q)\}_{n\geq 0}$ is $2$-$q$-log-convex.

For (ii), it follows from the conditions $g_k(q)g_{k+1}(q)\geq_q
h_{k+1}(q)$ and
\begin{eqnarray*}
&&g_k(q)g_{k+1}(q)g_{k+2}(q)g_{k+3}(q)
-g_{k+2}(q)g_{k+3}(q)h_{k+1}(q)-g_k(q)(q)g_{k+3}(q)h_{k+2}(q)-\nonumber\\
&& g_k(q)g_{k+1}(q)h_{k+3}(q) + h_{k+1}(q)h_{k+3}(q)\geq_q 0
\end{eqnarray*} for $k\geq0$ that $\mathcal {J}_n(q)$ is $q$-TP$_4$ by Lemma \ref{lem+diag+positivity} (ii), which again implies that the
Hankel matrix $\mathcal {H}$ is $q$-TP$_4$ by Claim \ref{cl1}. Thus,
by Proposition \ref{prop+higher+q-log-concave}, we have
$\{A_{n,0}(q)\}_{n\geq 0}$ is $3$-$q$-log-convex. \qed

\section{Applications of Theorem \ref{thm+2+q+log-convex}
} In this section, we will give some applications of Theorem
\ref{thm+2+q+log-convex}.
\subsection{Reduced form of Theorem \ref{thm+2+q+log-convex}}

For Theorem~\ref{thm+2+q+log-convex}, if let $q$ be a fixed
nonnegative real number, then the triangle $[A_{n,k}(q)]_{n,k\geq0}$
turns out to be one triangle of nonnegative real numbers. Thus, the
next result is a special case of Theorem~\ref{thm+2+q+log-convex}.
\begin{thm}\label{thm+2-log-convex}
Given nonnegative sequences $\{g_n\}_{n\geq0}$ and
$\{h_n\}_{n\geq0}$, assume that the triangular array
$[A_{n,k}]_{n,k\geq0}$ satisfies the recurrence
\begin{eqnarray*}
A_{n,k}&=& A_{n-1,k-1}+g_k\,A_{n-1,k}+h_{k+1}\,A_{n-1,k+1},\\
A_{n,0}&=& g_0\,A_{n-1,0}+h_{1}\,A_{n-1,1}
\end{eqnarray*}
 for $n\geq 1$ and
$k\geq 1$, where $A_{0,0}=1$, $A_{0,k}=0$ for $k>0$.
\begin{itemize}
\item [\rm (i)]
If $g_kg_{k+1}g_{k+2}- h_{k+1}g_{k+2}-g_{k}h_{k+2}\geq 0$ for
$k\geq0$, then the sequence $\{A_{n,0}\}_{n\geq0}$ is
$2$-log-convex.
\item [\rm (ii)]
If $g_kg_{k+1}\geq h_{k+1}$ and
\begin{eqnarray*}
g_kg_{k+1}g_{k+2}g_{k+3} -g_{k+2}g_{k+3}h_{k+1}-g_kg_{k+3}h_{k+2}-
g_kg_{k+1}h_{k+3} + h_{k+1}h_{k+3}\geq 0
\end{eqnarray*} for $k\geq0$, then the sequence $\{A_{n,0}\}_{n\geq0}$ is $3$-log-convex.
\end{itemize}
\end{thm}
\subsection{3-q-log-convexity of generating functions of rows of Riordan arrays}

Riordan arrays are very important in combinatorics. The literature
about Riordan arrays is vast and still growing and the applications
cover a wide range of subjects, such as enumerative combinatorics,
combinatorial sums, recurrence relations and computer science, among
other topics \cite{H13,SGWW91}. Recall that the Riordan array,
denoted by $(g(x), f(x))=[R_{n,k}]_{n,k\geq0}$, is an infinite lower
triangular matrix whose generating function of the $k$th column is
$x^kf^k(x)g(x)$ for $k\geq 0$, where $g(0)=1$ and $f(0)\neq 0$. It
can also be characterized by two sequences $\{a_n\}_{n\geq0}$ and
$\{z_n\}_{n\geq 0}$ such that $$R_{0,0}=1, R_{n+1,0}=\sum_{j\geq0}
z_jR_{n,j}, R_{n+1,k+1} =\sum_{j\geq0} a_jR_{n,k+j},$$ for $n,k\geq
0$.

In fact, there is a close relation between the triangular array
$[A_{n,k}(q)]_{n,k\geq0}$ in (\ref{re}) and the Riordan arrays. If
$g_k=g$ and $h_k=h$ for $k\geq1$, then the triangular array
$[A_{n,k}(q)]_{n,k\geq0}$ turns to be a kind of special interesting
Riordan arrays, denoted by $[A_{n,k}]_{n,k\geq0}$, which is defined
recursively:
\begin{eqnarray*}
A_{0,0}&=&1, \quad A_{0,k}=0\quad(k>0),\\
A_{n,0}&=&e\,A_{n-1,0}+h\,A_{n-1,1},\\
A_{n,k}&=&A_{n-1,k-1}+g\,A_{n-1,k}+h\,A_{n-1,k+1}\quad (n\geq 1).
\end{eqnarray*}
In fact, the above Riordan arrays $[A_{n,k}]_{n,k\geq0}$, or more
general triangular array $[A_{n,k}(q)]_{n,k\geq0}$ contains many
famous combinatorial sequences or classical triangular arrays in
combinatorics in a unified approach. The following are several basic
examples.
\begin{ex}\label{exm}
(1) The Catalan triangle of Aigner \cite{Aig991} is
$$C=[C_{n,k}]_{n,k\geq0}=\left[
      \begin{array}{rrrrr}
        1 &  &  &  &  \\
        1 & 1 &   &   &\\
        2 & 3 & 1 &   & \\
        5 & 9 & 5 & 1 &   \\
       \vdots &  &  &  & \ddots \\
      \end{array}
    \right],$$
where $C_{n+1,k}=C_{n,k-1}+2\,C_{n,k}+C_{n,k+1}$ and
$C_{n+1,0}=C_{n,0}+C_{n,1}$. The numbers in the first column are the
Catalan numbers $C_n$. \\
(2) The Catalan triangle of Shaprio \cite{Sha76} is
$$C'=[C'_{n,k}]_{n,k\geq0}=\left[
      \begin{array}{rrrrr}
        1 &  &  &  &  \\
        2 & 1 &   &   &\\
        5 & 4 & 1 &   & \\
        14 & 14 & 6 & 1 &   \\
       \vdots &  &  &  & \ddots \\
      \end{array}
    \right],$$
where $C'_{n+1,k}=C'_{n,k-1}+2\,C'_{n,k}+C'_{n,k+1}$ for $k\geq0$. The numbers in the first column are the Catalan numbers $C_n$. \\
(3) The large Schr\"{o}der triangle \cite{CKS12} is
 $$s=[s_{n,k}]_{n,k\geq0}=\left[
 \begin{array}{rrrrr}
 1 &  &  &  &  \\
 2 & 1 &  &  &  \\
 6 & 4 & 1 &  &  \\
 22 & 16 & 6 & 1 &  \\
 \vdots &  &  &  & \ddots \\
 \end{array}
 \right],$$ where $s_{n+1,k}=s_{n,k-1}+2\,s_{n,k}+2s_{n,k+1}$ and
$s_{n+1,0}=s_{n,0}+2s_{n,1}$. The numbers in the first column are
the large Schr\"{o}der numbers $S_n$.
\end{ex}
In fact, we have known that the generating functions of rows in many
combinatorial triangles, for example, Narayana triangles of Types
$A$ and $B$, the Stirling triangle of the second kind and the
Eulerian triangles, have $q$-log-convexity \cite{Zhu13} and even
stronger $3$-$q$-log-convexity (see Proposition \ref{prop+exam}).
Thus, it is natural to consider the similar property of the Riordan
arrays. In \cite{Zhu17}, we stated a criterion for $q$-log-convexity
of generating functions of rows of the Riordan array.
 As an extension, a criterion for $3$-$q$-log-convexity can be given as follows.
\begin{thm}\label{thm+riordan+q-lcx}
For $g\geq e\geq0$ and $h\geq0$, assume that the Riordan array
$[A_{n,k}]_{n,k\geq0}$ satisfies the recurrence
\begin{eqnarray}\label{eqq}
A_{n,k}&=&A_{n-1,k-1}+g\,A_{n-1,k}+h\,A_{n-1,k+1} ,\\
A_{n,0}&=&e\, A_{n-1,0}+h\, A_{n-1,1}\nonumber,
\end{eqnarray}
for $n\geq 1$ and $k\geq 1$, where $A_{0,0}=1$, $A_{0,k}=0$ for
$k>0$.  Let the generating functions $\mathscr{A}_n(q)=\sum_{k\geq
0}A_{n,k} q^k$ for $n\geq 0$. If $ge\geq rh$, then
$\{\mathscr{A}_n(q)\}_{n\geq0}$ is $r$-q-log-convex for $r=2,3$
respectively.
\end{thm}

\begin{proof}
 Define a new
triangular array $\mathcal {S}=[s_{n,k}(q)]_{n,k\geq0}$, where
\begin{eqnarray*}
s_{n,k}(q)=\sum_{i\geq k}A_{n,i} q^{i-k}\text { for all $n$ and
$k$.}
\end{eqnarray*}
Thus, by (\ref{eqq}), we deduce that the triangular array $\mathcal
{S}=[s_{n,k}(q)]_{n,k\geq0}$ satisfies the recurrence relation
\begin{eqnarray*}
s_{n,k}(q)&=&\,s_{n-1,k-1}(q)+g\,s_{n-1,k}(q)+h\,s_{n-1,k+1}(q) \quad (n\geq1, k\geq 1),\nonumber\\
s_{n,0}(q)&=&(e+q)\,s_{n-1,0}(q)+[(g-e)q+h]\,s_{n-1,1}(q) \quad
(n\geq1),
\end{eqnarray*}
where $s_{0,0}=1$, $s_{0,k}=0$ for $k>0$. Obviously,
$\mathscr{A}_n(q)=s_{n,0}(q).$

For $r=2$, it follows from $g^2\geq ge\geq 2h$ that
$$g^2(e+q)-[(g-e)q+h]g-(e+q)h=(ge-h)q+eg^2-gh-eh\geq_q0$$ and
$$g^3-hg-gh=g(g^2-2h)\geq0.$$
Thus we get the $2$-$q$-log-convexity of
$\{\mathscr{A}_n(q)\}_{n\geq0}$ by Theorem \ref{thm+2+q+log-convex}
(i).

For $r=3$, it follows from $g^2\geq ge\geq 3h$ that
\begin{eqnarray*}
&&g^3(e+q)-[(g-e)q+h]g^2-2gh(e+q)+h[(g-e)q+h]\\
&=&[eg^2-2gh+h(g-e)]q+g^3e-g^2h-2ghe+h^2\\
&\geq_q&[g(eg-2h)+h(g-e)]q+g^2(ge-3h)+h^2\\
&\geq_q&0
\end{eqnarray*} and
$$g^4-3hg^2+h^2=g^2(g^2-3h)+h^2\geq0.$$
Hence we obtain the $3$-$q$-log-convexity of
$\{\mathscr{A}_n(q)\}_{n\geq0}$ by Theorem \ref{thm+2+q+log-convex}
(ii). So we complete the proof.
\end{proof}

The following  result is immediate from Theorem
\ref{thm+riordan+q-lcx}.
\begin{prop}
The generating functions of rows in each triangle of Example
\ref{exm} form a $3$-$q$-log-convex sequence.
\end{prop}
It is known for the generating function for the sequence
$\{A_{n,0}\}_{n\geq0}$ in (\ref{eqq}) that
\begin{eqnarray*}
\sum_{n\geq 0}
A_{n,0}x^n=\frac{1-[2a-s]x-\sqrt{1-2sx+[s^2-4t]x^2}}{2[s-a]x+2[a^2-as+t]x^2},
\end{eqnarray*}
see \cite{Aig08}. Thus, by Theorem \ref{thm+riordan+q-lcx}, we have
the following criterion for 3-$q$-log-convexity in terms of the
generating functions.
\begin{prop}\label{prop+func+3+q+log-convex}
Assume that the generating function
\begin{eqnarray}\label{fuction} \sum_{n\geq 0}
F_{n}(q)x^n=\frac{1-[2a(q)-s(q)]x-\sqrt{1-2s(q)x+[s(q)^2-4t(q)]x^2}}{2[s(q)-a(q)]x+2[a(q)^2-a(q)s(q)+t(q)]x^2},
\end{eqnarray}
where polynomials $s(q)\geq_q a(q)\geq_q 0$ and $t(q)\geq_q 0$. Then
we have the following results:
\begin{itemize}
\item [\rm (i)]
If $a(q)s(q)\geq_q 2t(q)$, then the sequence $\{F_{n}(q)\}_{n\geq
0}$ is $2$-$q$-log-convex.
\item [\rm (ii)]
If $a(q)s(q)\geq_q t(q)$ and $a(q)s^3(q) - 2a(q)s(q)t(q) - s^2(q)
t(q) + t^2(q)\geq_q0$, then the sequence $\{F_{n}(q)\}_{n\geq 0}$ is
$3$-$q$-log-convex.
\end{itemize}
\end{prop}

\section{Applications of Theorem \ref{thm+q+continued+q}}

In this section, we will apply Theorem \ref{thm+q+continued+q} to
some classical examples in combinatorics.
\subsection{$2$-$q$-log-convexity of alternating Eulerian polynomials}
Let $\mathcal{S}_n$ denote the symmetric group of all
permutations of $[n]$. Similar to the descent statistic of
$\mathcal{S}_n$, the number of alternating descents of a permutation
$\pi\in\mathcal{S}_n$ is defined by
$$altdes(\pi) = |\{2i : \pi(2i)<\pi(2i+1)\}
\cup\{ 2i + 1 : \pi(2i + 1) > \pi(2i + 2)\}|.$$ We say that $\pi$
has a $3$-descent at index $i$ if $\pi(i)\pi(i+1)\pi(i+2)$ has one
of the patterns: $132$, $213$, or $321$. Chebikin \cite{Ch08} proved
that the alternating descent statistic of permutations in
$\mathcal{S}_n$ is equidistributed with the $3$-descent statistic of
permutations in $\{\pi\in\mathcal{S}_{n+1}:\pi{_1}=1\}$. Then
alternating Eulerian polynomials $A^{*}_n(x)$ and the alternating
Eulerian numbers $A^*_{n,k}$ are defined as follows:
$$A^{*}_n(x)=\sum_{\pi\in\mathcal{S}_n}x^{altdes(\pi)}=\sum_{k=0}^{n-1}A^*_{n,k}x^k,$$
where the first few items $A^{*}_1(x)=1$, $A^{*}_2(x)=1+x$ and
$A^{*}_3(x)=2+2x+2x^2$. It was proved that
$$\sum_{n\geq1}A^{*}_n(x)\frac{z^n}{n!}=\frac{sec(1-x)z + tan(1-x)z
-1}{1-x[sec(1-x)z + tan(1-x)z]}.$$
\begin{prop}
Let $A^{*}_n(q)$ be the alternating Eulerian polynomials for
$n\geq0$. Then
\begin{itemize}
\item [\rm (i)] we have
\begin{eqnarray*}
\sum\limits_{n=0}^{\infty}A^{*}_{n+1}(q) x^n=\DF{1}{1-
g_0x-\DF{h_1x^2}{1- g_1x-\DF{h_2x^2}{1- g_2x-\ldots}}},
\end{eqnarray*}
where $g_i=(i+1)(q+1)$ and $h_i=(q^2+1)i(i+1)/2$ for $i\geq0$.
\item
[\rm (ii)] The sequence $\{A^{*}_{n+1}(q)\}_{n\geq0}$ is
$2$-$q$-log-convex.
\end{itemize}
\end{prop}
\begin{proof}
Let $D_x$ denote the differential operator $\frac{d}{ dx}$. In 1995,
Hoffman considered derivative polynomials defined respectively by
$D^n_x(tan(x)) = P_n(tan(x))$ for $n\geq 0$, where $P_0(q) = q$ and
$P_{n+1}(q) = (1 + q^2)P'_n (q)$. In \cite{V14}, it was proved that
the generating function of $P_n(q)$ has the following continued
fraction expansion:
\begin{eqnarray}\label{fraction+alte+Eulr}
\sum\limits_{n=0}^{\infty}\frac{P_{n+1}(q)}{1+q^2} x^n=\DF{1}{1-
g_0x-\DF{h_1x^2}{1- g_1x-\DF{h_2x^2}{1- g_2x-\ldots}}},
\end{eqnarray}
where $g_i=2(i+1)q$ and $h_i=(1+q^2)i(i+1)$ for $i\geq0$.

On the other hand, Ma and Yeh \cite{MY16} recently proved
$$2^n(1+x^2)A^{*}_n(x) = (1-x)^{n+1}P_n(\frac{1+x}{1-x})$$ for
$n\geq1$. Thus by (\ref{fraction+alte+Eulr}), we have
\begin{eqnarray*}
&&\sum\limits_{n=0}^{\infty}2^nA^{*}_{n+1}(q) (x/2)^n\\
&=&\sum\limits_{n=0}^{\infty}\frac{(1-q)^{n+2}}{2(1+q^2)}P_{n+1}(\frac{1+q}{1-q}) (x/2)^n\\
&=&\sum\limits_{n=0}^{\infty}\frac{P_{n+1}(\frac{1+q}{1-q})}{1+(\frac{1+q}{1-q})^2} [\frac{(1-q)x}{2}]^n\\
&=&\DF{1}{1- g_0x-\DF{h_1x^2}{1- g_1x-\DF{h_2x^2}{1- g_2x-\ldots}}},
\end{eqnarray*}
where $g_i=(i+1)(q+1)$ and $h_i=(q^2+1)i(i+1)/2$ for $i\geq0$. This
proves that (i) holds.

(ii) Because $g_i=(i+1)(q+1)$ and $h_i=(q^2+1)i(i+1)/2$ for
$i\geq0$, we have
\begin{eqnarray*}
&&g_i(q)g_{i+1}(q)g_{i+2}(q)-
h_{i+1}(q)g_{i+2}(q)-g_{i}(q)h_{i+2}(q)\\
&=&(i+3)(i+2)(i+1)(1+q)^3-(1+q)(q^2+1)(i+3)(i+2)(i+1)/2\\
&&-(q+1)(q^2+1)(i+3)(i+2)(i+1)/2\\
&=&2q(q+1)(i+3)(i+2)(i+1)\\
&\geq_q& 0
\end{eqnarray*}
Thus, by Theorem \ref{thm+q+continued+q} (i), we obtain that
$\{A^{*}_{n}(q)\}_{n\geq1}$ is $2$-$q$-log-convex.
\end{proof}
\begin{rem}
For $\{A^{*}_{n}(q)\}_{n\geq1}$,  it is easy to get that the
condition (ii) in Theorem \ref{thm+q+continued+q} does not follow.
In fact, $[A^{*}_{i+j}(q)]_{i,j\geq1}$ is not $q$-TP$_4$ because
\begin{eqnarray*}\left|\begin{array}{cccc}
A^{*}_{2}(q)&A^{*}_{3}(q)&A^{*}_{4}(q)&A^{*}_{5}(q)\\
A^{*}_{3}(q)&A^{*}_{4}(q)&A^{*}_{5}(q)&A^{*}_{6}(q)\\
A^{*}_{4}(q)&A^{*}_{5}(q)&A^{*}_{6}(q)&A^{*}_{7}(q)\\
A^{*}_{5}(q)&A^{*}_{6}(q)&A^{*}_{7}(q)&A^{*}_{8}(q)\\
\end{array}\right|=-324 + 1296 q + 2592 q^2 + \cdots +
    1296 q^{15} - 324 q^{16}.
\end{eqnarray*}
 But it seems that $\{A^{*}_{n}(q)\}_{n\geq1}$
is still $3$-$q$-log-convex by the computation from Mathematica.
\end{rem}

\subsection{$2$-log-convexity
of Euler numbers}  Recall the definition of Euler numbers as
follows. A permutation $\pi$ of the $n$-element set
$[n]=\{1,2,\ldots,n\}$ is {\it alternating} if
$\pi(1)>\pi(2)<\pi(3)>\pi(4)<\cdots \pi(n)$. The number $E_n$ of
alternating permutations of $[n]$ is known as an {\it Euler number}.
The sequence has the exponential generating function
$$\sum_{k=0}^{n}E_n\frac{x^n}{n!}=\tan x+\sec x.$$ Liu and Wang
\cite{LW07} proved that $\{E_n\}_{n\geq0}$ is log-convex by
 Davenport-P\'olya Theorem. The next result
strengthens the log-convexity of Euler numbers.
\begin{prop}
The Euler numbers form a $2$-log-convex sequence.
\end{prop}

\begin{proof}
In \cite{V14}, it was proved that
\begin{eqnarray*}
\sum\limits_{n=0}^{\infty}E_{n+1} x^n=\DF{1}{1- g_0x-\DF{h_1x^2}{1-
g_1x-\DF{h_2x^2}{1- g_2x-\ldots}}},
\end{eqnarray*}
where $g_i=i+1$ and $h_i=i(i+1)/2$ for $i\geq0$. Note that
$$g_ig_{i+1}g_{i+2}- h_{i+1}g_{i+2}-g_{i}h_{i+2}= 0$$ for $i\geq0$.
Hence the $2$-log-convexity of $\{E_{n+1}\}_{n\geq0}$ follows from
Theorem \ref{thm+q+continued+q} (i).
\end{proof}

\begin{rem}
For $\{E_{n+1}\}_{n\geq0}=\{1, 2, 5, 16, 61, 272, 1385, 7936, 50521,
353792,\cdots\}$, then
\begin{eqnarray*}
g_kg_{k+1}g_{k+2}g_{k+3} -g_{k+2}g_{k+3}h_{k+1}-g_kg_{k+3}h_{k+2}-
g_kg_{k+1}h_{k+3} + h_{k+1}h_{k+3}=-6
\end{eqnarray*} for $k=0$ and
\begin{eqnarray*}\left|\begin{array}{cccc}
E_{n}&E_{n+1}&E_{n+2}&E_{n+3}\\
E_{n+1}&E_{n+2}&E_{n+3}&E_{n+4}\\
E_{n+2}&E_{n+3}&E_{n+4}&E_{n+5}\\
E_{n+3}&E_{n+4}&E_{n+5}&E_{n+6}\\
\end{array}\right|=-324,-154224
\end{eqnarray*}
for $n=2,4$, respectively. Thus the condition (ii) in Theorem
\ref{thm+q+continued+q} does not follow and we can't get
$3$-log-convexity of $\{E_{n+1}\}_{n\geq0}$. But it seems that
$\{E_{n+1}\}_{n\geq0}$ is still $3$-log-convex by the computation
from Mathematica.
\end{rem}

\subsection{$3$-$q$-log-convexity of  classical combinatorial polynomials}
It is known that many important combinatorial polynomials, such as
the Bell polynomials, Eulerian polynomials, and Narayana polynomials
of types A and B, and so on, are $q$-log-convex. By Theorem
\ref{thm+2+q+log-convex}, we have the following stronger result in a
unified approach.

\begin{prop}\label{prop+exam}
The six sequences of polynomials in Example \ref{basic-qSM} are all
$3$-$q$-log-convex.
\end{prop}
\begin{proof}
Because the proofs are similar, for brevity, we only present the
proof of the $3$-$q$-log-convexity of the Bell polynomials and omit
the proofs of others.

Notice for the case of the Bell polynomials that $g_k(q)=q+k$ and
$h_k(q)=kq$ for $k\geq 0$. It is obvious that
$$g_k(q)g_{k+1}(q)- h_{k+1}(q)=q^2+kq+k^2+k\geq_q0.$$ Thus, in order
to prove the $3$-$q$-log-convexity of the Bell polynomials, by
Theorem \ref{thm+2+q+log-convex} (ii), it suffices to check
\begin{eqnarray*}
&&g_k(q)g_{k+1}(q)g_{k+2}(q)g_{k+3}(q)
-g_{k+2}(q)g_{k+3}(q)h_{k+1}(q)-g_k(q)g_{k+3}(q)h_{k+2}(q)\nonumber\\
&&- g_k(q)g_{k+1}(q)h_{k+3}(q) +
h_{k+1}(q)h_{k+3}(q)\\
&=&k^4+(6+q)k^3+(11+3q+q^2)k^2+(q^3+q^2+2q+6)k+q^4\\
&\geq_q&0,
\end{eqnarray*}
as desired. This proof is complete.
\end{proof}
\begin{rem}
In fact, the Hankel matrix $[A_{i+j,0}(q)]_{i,j\geq0}$ for each
sequence of polynomials in Example \ref{basic-qSM} are $q$-TP
\cite{WZ16}. So we can also obtain the $3$-$q$-log-convexity by
Proposition \ref{prop+higher+q-log-concave}.
\end{rem}

\subsection{$3$-$q$-log-convexity from the Stieltjes continued fractions}
Let $\{S_n(q)\}_{n\geq 0}$ be a sequence of polynomials with
$S_0(q)=1$. The expansion
\begin{eqnarray}\label{st+fract}
\sum\limits_{i=0}^{\infty}S_n(q)x^i=\DF{1}{1-\DF{t_1(q)x}{1-\DF{t_2(q)x}{1-\ldots}}}
\end{eqnarray}
is called the Stieltjes continued fraction. In view of the following
famous contraction formulae
\begin{eqnarray*}
\DF{1}{1-\DF{c_1x}{1-\DF{c_2x}{1-\ldots}}}
&=&\DF{1}{1-c_1x-\DF{c_1c_2x^2}{1-(c_2+c_3)x-\DF{c_3c_4x^2}{1-\ldots}}},
\end{eqnarray*}
we deduce the Jacobi continued fraction expansion
\begin{eqnarray*}
\sum\limits_{i=0}^{\infty}S_n(q)x^i=\DF{1}{1-t_1(q)x-\DF{t_1(q)t_2(q)x^2}{1-[t_2(q)+t_3(q)]x-\DF{t_3(q)t_4(q)x^2}{1-\ldots}}}.
\end{eqnarray*}

 \begin{thm}\label{thm+st}
If $t_n(q)$ are polynomials in $q$ for $n\geq1$ and $t_n(q)\geq_q0$,
then the sequence $\{S_n(q)\}_{n\geq0}$ defined in (\ref{st+fract})
is $3$-$q$-log-convex.
 \end{thm}
\begin{proof}
Noting that $g_i=t_{2i+1}(q)+t_{2i}(q)$ and
$h_i=t_{2i-1}(q)t_{2i}(q)$, we obtain $$g_ig_{i+1}-
h_{i+1}=t_{1+2i}(q)t_{3+2i}(q)+t_{2i}(q)t_{2+2i}(q)+t_{2i}(q)t_{3+2i}(q)\geq_q0$$
and
\begin{eqnarray*}
&&g_ig_{i+1}g_{i+2}g_{i+3} -g_{i+2}g_{i+3}h_{i+1}-g_ig_{i+3}h_{i+2}-
g_ig_{i+1}h_{i+3} +
h_{i+1}h_{i+3}\\
&=&t_{2i}(q)\left[t_{3+2i}(q)t_{5+2i}(q)t_{7+2i}(q)+t_{2+2i}(q)\left[t_{5+2i}(q)t_{7+2i}(q)
+t_{4+2i}(q)\left(t_{6+2i}(q)+t_{7+2i}(q)\right)\right]\right]\\
&&+t_{1+2i}(q)t_{3+2i}(q)t_{5+2i}(q)t_{7+2i}(q)\\
 &\geq_q&0.
\end{eqnarray*}
Thus, by Theorem \ref{thm+q+continued+q}, the sequence
$\{S_n(q)\}_{n\geq0}$ defined in (\ref{st+fract}) is
$3$-$q$-log-convex.
\end{proof}

\begin{rem}
In fact, by \cite[Lemma 3.3]{WZ16}, the Hankel matrix
$[S_{i+j,0}(q)]_{i,j\geq0}$ for the sequence $\{S_n(q)\}_{n\geq0}$
defined in (\ref{st+fract}) is $q$-TP. That is to say that
$\{S_n(q)\}_{n\geq0}$ defined in (\ref{st+fract}) is a $q$-Stieltjes
sequences. For $q$ being nonnegative real numbers, this particularly
gives a new proof of Stieltjes sequences.
\end{rem}

The elliptic functions $cn$ and $dn$ are defined by
$$cn(u,\alpha)=cos am(u,\alpha);
\quad dn(u,\alpha)=\sqrt{1-{\alpha}^2\sin^2 am(u,\alpha)},$$ where
$am(u,\alpha)$ is the inverse of an elliptic integral: by definition
$$am(u,\alpha)=\phi\quad iff\quad u=\int_0^{\phi}\frac{d {\it t}}{\sqrt{1-{\alpha}^2\sin^2
t}}.$$

The function $cn(u,\alpha)$ and $dn(u,\alpha)$ have power series
expansions:
\begin{eqnarray*}
cn(u,\alpha)&=& \sum_{n\geq0}(-1)^{n-1}
c_n({\alpha}^2)\frac{u^{2u}}{2n!};\\
dn(u,\alpha)&=& \sum_{n\geq0}(-1)^{n-1}
d_n({\alpha}^2)\frac{u^{2u}}{2n!},
\end{eqnarray*}
where $c_n(q)$ and $d_n(q)$ are polynomials of degree $n-1$ with
$d_n(q)$ the reciprocal polynomial of $c_n(q)$. In addition,
Flajolet \cite[Theorem 4]{Fla80} proved that the coefficient
$c_{n,k}$ of the polynomial $c_n(q)$ counts the alternating
permutations over $[2n]$ having $k$ minima of even value and
\begin{eqnarray}
\sum\limits_{i=0}^{\infty}c_n(q)x^i=\DF{1}{1-\DF{1^2x}{1-\DF{2^2qx}{1-\DF{3^2x}{1-\DF{4^2qx}\ldots}}}}.
\end{eqnarray}
Thus the following proposition is an immediate consequence of
Theorem \ref{thm+st}.
\begin{prop}
The polynomials $c_n(q)$ form a $3$-$q$-log-convex sequence.
\end{prop}

\section{Concluding remarks and open problems }\hspace*{\parindent}
In this paper we have given some sufficient conditions for the
$2$-$q$-log-convexity and $3$-$q$-log-convexity. Finally, using our
results we show the $2$-$q$-log-convexity or $3$-$q$-log-convexity
of more classical combinatorial sequences of polynomials and
numbers.

Based on the result in Theorem \ref{thm+q+continued+q}, it is
natural to consider the following problem for the infinite
$q$-log-convexity of Boros-Moll polynomial \cite{Chen08}.
\begin{prob}
Can we obtain a continued fraction expansion for the generating
function of the sequence of Boros-Moll polynomials ?
\end{prob}

By Proposition \ref{prop+exam}, we have known that each sequence of
polynomials in Example \ref{basic-qSM} is $3$-$q$-log-convex. In
addition, for any fixed nonnegative real number $q$, each sequence
of polynomials in Example \ref{basic-qSM} is infinitely log-convex
\cite{WZ16}. Thus, we conclude this paper with the following
conjecture concerning the infinite $q$-log-convexity for further
research.

\begin{conj}
Every sequence of polynomials in Example \ref{basic-qSM} is
infinitely $q$-log-convex.
\end{conj}

Note a criterion for infinite log-convexity that if the infinite
Hankel matrix $[x_{i+j}]_{i,j\geq0}$ is TP, then so is
$[x_{i+j+2}x_{i+j}-x_{i+j+1}^2]_{i,j\geq0}$, see \cite{WZ16}. Thus,
Proposition \ref{prop+higher+q-log-concave} appears to indicate the
following conjecture, which is a possible approach to the higher
order $q$-log-convexity.
\begin{conj}\label{conj+tp}
Let $r$ be any positive integer and $\{x_n(q)\}_{n\geq0}$ be a
sequence of polynomials. If the infinite Hankel matrix
$[x_{i+j}(q)]_{i,j\geq0}$ is $q$-TP$_{r+1}$, then
$$[x_{i+j+2}(q)x_{i+j}(q)-x_{i+j+1}(q)^2]_{i,j\geq0}$$ is
$q$-TP$_{r}$.
\end{conj}
This Conjecture \ref{conj+tp} may be proved by the combinatorial
approach from the planar network, see \cite{Bre95,GV85,Sag921,S90}
for details.

\section{Acknowledgements}
The author would like to thank the anonymous reviewer for many
valuable remarks and suggestions to improve the original manuscript.
In addition, results of this paper were presented in the Institute
of Mathematics, Academia Sinica, Taipei in Jan. 2016, the seventh
National Conference on Combinatorics and Graph Theory at Hebei
Normal University in Aug. 2016 and at Nakai University in Oct. 2016,
respectively.



\end{document}